\documentclass[twoside,leqno]{article}
\usepackage{verbatim}
\usepackage{amssymb}
\usepackage{amsbsy}
\usepackage{amscd}
\usepackage{amsmath}
\usepackage{amsthm}
\usepackage[mathscr]{eucal}

\usepackage{algorithmic}
\usepackage{comment}
\usepackage{color}
\usepackage{soul}
\usepackage{setspace}
\usepackage{graphicx}

\newtheorem{theorem}{Theorem}\numberwithin{theorem}{section}

\newtheorem{proposition}[theorem]{Proposition}
\newtheorem{definition}[theorem]{Definition}

\newtheorem{question}[theorem]{Question}
\newtheorem{problem}[theorem]{Problem}
\newtheorem{remark}[theorem]{Remark}

\newtheorem{example}[theorem]{Example}
\newtheorem{algorithm}[theorem]{Algorithm}

\newenvironment{M2}{ \begin{spacing}{0.4} %\begin{quote}
\bigskip \small} { %\end{quote}
\end{spacing} 
\bigskip 
}

\newcommand{\B}[1]{\mathbb #1}
\newcommand{\C}[1]{\mathcal #1}
\newcommand{\F}[1]{\mathfrak #1}

\newcommand{\alg}[1]{\operatorname{#1}}
\newcommand{\ideal}[1]{\langle #1 \rangle}

\DeclareMathOperator{\initial}{in}
\DeclareMathOperator{\NF}{NF}
\DeclareMathOperator{\lcm}{lcm}
\DeclareMathOperator{\im}{im}
\DeclareMathOperator{\Syz}{Syz}

\newcommand{\Inc}{\operatorname{Inc}(\B N)}
\newcommand{\IncT}{\operatorname{Inc}^\tau(\B N)}
\newcommand{\SymN}{\F S_\infty}

\newcommand{\mon}{M}
\newcommand{\Sym}{\F S_\infty}
\newcommand{\LT}{\initial_{\leq}}

\newcommand{\LC}{\operatorname{LC}}
\newcommand{\LM}{\operatorname{LM}}

\newcommand{\GB}{{Gr\"obner basis}}
\newcommand{\GBs}{{Gr\"obner bases}}
\newcommand{\EGB}{{equivariant Gr\"obner basis}}
\newcommand{\EGBs}{{equivariant Gr\"obner bases}}

\title{\Large Equivariant Gr\"obner bases\thanks{Part of this work took place during the ``Free Resolutions, Representations, and Asymptotic Algebra'' workshop at the BANFF International Research Station, April, 2016.}}
\author{Christopher J. Hillar \\
\and
Robert Kroner
\and
Anton Leykin\thanks{Research of AL is supported in part by NSF grant DMS-1151297.}}

\begin{document}

\maketitle 

\begin{abstract}
Algorithmic computation in polynomial rings is a classical topic in mathematics. However, little attention has been given to the case of rings with an infinite number of variables until recently when theoretical efforts have made possible the development of effective routines. Ability to compute relies on finite generation up to symmetry for ideals invariant under a large group or monoid action, such as the permutations of the natural numbers. We summarize the current state of theory and applications for equivariant Gr\"obner bases, develop several algorithms to compute them, showcase our software implementation, and close with several open problems and computational challenges.
\end{abstract}

%\tableofcontents
\section{Introduction}

\subsection{History}

The theory of polynomial rings is an old and well-studied subject.  However, as far as we can tell, a rigorous set of tools for algorithmic computation in such rings was only first developed starting in 1913 \cite{gjunter1913} by the Russian/Soviet mathematician N.~Gjunter.  This project culminated with Gjunter's review of the theory in 1941 \cite{gunther1941modules} but went unnoticed until recently~\cite{renschuch2003contributions}.  Outside of this rather newly discovered reference, general algorithmic theory in (possibly non-commutative) rings has a long history of independent thinkers.  For instance, the work \cite{bergman1978diamond} (see also \cite{bokut1976embeddings} as influenced by \cite{shirshov1962some}) was inspired by an algorithmic proof of the Poincar\'e-Birkhoff-Witt theorem.  

Attribution of an algorithmic theory of polynomial rings and ideals is usually given to Buchberger \cite{buchberger1965algorithmus}, who named the main tools ``Gr\"obner bases" after his Ph.D. advisor. Hironaka \cite{hironaka1964resolution} used a similar concept called ``standard bases" in power series rings to prove his theorem on resolutions of singularities.  

The main consequences of these projects are effective procedures for polynomial equation solving over fields such as the complex numbers~$\mathbb{C}$.  Practical questions of ideal membership or equation feasibility were now answerable (provably) using a finite programmable set of computations.  

Since these early efforts, much progress has been made on the mathematical and computational theory of polynomial algebra involving a finite number of indeterminates.  In this article, we consider computation in rings with infinite numbers of indeterminates, a topic that is part of a burgeoning new field called ``asymptotic algebra".  At first, such a concept seems at odds with the non-Noetherianity of even simple ideals such as the maximal ideal: \[I = \langle x_0, x_1, \ldots  \rangle \subset \mathbb C[x_0, x_1,\ldots].\]

However, if extra structure is imposed on the class of ideals under consideration, such as a large group action, then it is possible to develop a theory of algorithmic computation.  For instance, the ideal $I$ above has a single generator up to the action of permuting indices on polynomials.

The concept of \emph{equivariant Gr\"obner bases} (EGB) was first used in an application to meta-abelian group theory \cite{cohen1967laws} and later developed into an algorithmic theory \cite{Emmott, Cohen87}. Similar to the story of Gr\"obner bases (in finitely many variables), the concept was rediscovered several decades later in \cite{aschenbrenner2007finite, aschenbrenner2008algorithm} and applied to solve in a unified manner several problems in algebraic statistics \cite{hillar2012finite}.  The theory was also useful in other applications such as those to algebraic chemistry \cite{Draisma08b} and asymptotic tensor geometry \cite{draisma2014bounded} (see \cite{draisma2014noetherianity} for an elegant survey of these techniques).

In the meantime, several works have started to make practical use of this effective computational machinery.  As a simple example, consider 
the following classical theorem in toric algebra that has been a starting point for several investigations into finiteness in asymptotic algebra.

\begin{theorem}\label{toric2x2}
Let $i > j$ run over natural numbers.  The kernel of $\mathbb C[y_{ij}] \to \mathbb C[x_i]$, $y_{ij} \mapsto x_i x_j$, is generated by the $2 \times 2$ minors (not containing diagonal entries) of the symmetric matrix $y$.
\end{theorem}

This result can be proved using equivariant Gr\"obner bases, as first demonstrated by J. Draisma with the following rather innocuous-looking Input/Output pair on a computer:
\begin{verbatim}
Input:  { y_{10} - x_1 x_0 }.
Output:  { x_0 x_1 - y_{10}, x_2 y_{10} - x_1 y_{20},
  x_2 y_{10} - x_0 y_{21}, x_1 y_{20} - x_0 y_{21},
  x_0^2 y_{21} - y_{20} y_{10},
  y_{32} y_{10} - y_{30} y_{21},  
  y_{31} y_{20} - y_{30} y_{21} }.
\end{verbatim}
Specifically, the ideal $\ker{(y_{ij} \mapsto x_i x_j)}$ for $i > j$ is generated by a finite set of $2 \times 2$ minors up to symmetry, 
which is witnessed by the last two polynomials of the EGB output above.

The first equivariant Gr\"obner basis computation to prove a new theorem that we are aware of occurs in \cite{Brouwer09e}. This leads us to the next subsection.

\subsection{Applications}

As a prelude, we start with an application of classical Gr\"obner bases that deals with a seemingly infinite problem; here, of course, a recurrence helps to ``control infinity''.  

The Fibonacci sequence $F_n = 1,1,2,3,5,\ldots$ ($n= 1, 2, \ldots$) is a \textit{strong divisibility sequence}, in that we have $\gcd(F_n, F_m) = F_{\gcd(n,m)}$; in particular, $F_m$ divides $F_n$ if $m \, | \, n$.  This surprising fact was used by \'Edouard Lucas for Mersenne prime testing.  Is there a direct explanation for the integrality of $F_{3n}/F_n \in \mathbb Z$?  It turns out that there is an identity:
\begin{equation}\label{F3n}
(F_{3n} - 5 F_n^3 - 3 F_n)(F_{3n} - 5 F_n^3 + 3 F_n) = 0,
\end{equation}
which explains in an explicit manner strong divisibility for this case.   Is it possible to use Gr\"obner bases to derive this relation?  The following Macaulay 2 code does exactly that:
\begin{M2}
\begin{verbatim}
i1: R = QQ[z, x, y, t, MonomialOrder => Eliminate 2]

i2: I = ideal(x + y - z, (x*z - y^2)^2 - 1, t - z^3 - y^3 + x^3)

i3: toString groebnerBasis I

o3 = matrix {{25*y^6-10*y^3*t-9*y^2+t^2, z-x-y, ...
\end{verbatim}
\end{M2}  
In this computation, the variables $z,x,y,t$ represent the recurrence values \mbox{$F_{n+1}, F_{n-1}, F_n, F_{3n}$,} respectively.  The first generator of $I$ defines the recurrence, the second is Cassini's identity, and the third is Lucas'.

One can check that factoring the first polynomial in the list above gives (\ref{F3n}).  Bootstrapping with extra equations, we can also discover that:
\begin{equation}\label{F5n}
(F_{5n} - 25 F_n^5 - 25 F_n^3 - 5 F_n)(F_{5n} - 25 F_n^5 + 25 F_n^3 - 5 F_n) = 0.
\end{equation}
In turn, these findings incite conjectures and proofs.  
For instance, we leave it to the reader to use modular arithmetic to verify from (\ref{F5n}) that the integer $\frac{F_{5n}}{5F_n}$ always has unit digit $1$ base ten.  
More generally, the following natural problem arises from this investigation:  Given $\ell$, find a nonzero polynomial $P(y, t) \in \mathbb Z[y, t]$ satisfying an identity of the form $P(F_n, F_{\ell n}) = 0$ (see \cite{HilLev:07} for more on ``polynomial recurrences").

The above is classical.  Here, we are interested in problems with not four or even twenty-four indeterminates, but rather an infinite number of them.  Take, for instance, the following basic ideal membership question.  Let $I \subset \mathbb C[x_0,x_1,\ldots]$ be the ideal generated by all permutations acting on the polynomial $f = x_0 x_1 - x_1 x_2^2 +x_1^2$.  Is the following in $I$?
\[ h = x_0 x_4^2 + x_0 x_1^2  +x_1 x_0^2 - 2 x_1 x_0 + x_0 x_3 x_4 - x_0 x_5^2 - x_0 x_3 x_5 - 2 x_1^2.\]

The difference in this question from classical problems of polynomial algebra is that \textit{a priori} there is no guarantee a particular computation, say, with a truncated polynomial ring $\mathbb C[x_0,x_1,\ldots, x_N]$ will do the job.  Nonetheless, the following code gives us an answer to our question \cite{EquivariantGB}.
\begin{M2}
\begin{verbatim}
i1: needsPackage "EquivariantGB"

i2: R = buildERing({symbol x}, {1}, QQ, 6);

i3: h = x_0*x_4^2+x_0*x_1^2+x_1*x_0^2-2*x_1*x_0+x_0*x_3*x_4- ...

i4: G = egb({x_0*x_1 - x_1*x_2^2 + x_1^2}, Algorithm=>Incremental)

       2       2      2           3      2     2    2         
o4 = {x x  - 2x  + x x  - 2x x , x  - x x , x x  - x  - x x , 
       1 0     1    1 0     1 0   1    1 0   2 0    1    1 0 

                    2           2 
      x x  - x x , x  + x x  - x  - x x }
       2 1    2 0   2    2 0    1    1 0

i5: reduce(h, G)

o5 = 0
\end{verbatim}
\end{M2}  
With the \EGB\ produced above, we can solve ideal membership problems and much more, just as we can use classical Gr\"obner bases in numerous applications.  %Note that arbitrarily large numbers of generators \cite{hillar2008minimal}.  

Developing the machinery to solve such questions is more than an intellectual curiosity.  Not only can basic facts now be proved by computer such as Theorem~\ref{toric2x2} but also cutting edge conjectures.  For example, using \cite{EquivariantGB}, it is possible to verify \cite{draisma2013noetherianity, Krone:egb-toric} the first nontrivial case of a basic finiteness conjecture for toric ideals \cite{aschenbrenner2007finite}.  

\begin{theorem}[Proved by computer]\label{monomthm}
For $n > 1$, let $I_n = \ker (y_{ij} \mapsto x_i^2 x_j)$, $1 \leq i \neq j \leq n$.  The invariant chain of toric ideals $I_2 \subset I_3 \subset \cdots$ stabilizes up to the symmetric group.  That is, there is some $N$ such that all elements of $I_m$, $m > N$, are polynomial consequences of relabellings of a finite set of generators of $I_N$.
\end{theorem}

%We note that whether there is always an equivariant Gr\"obner bases is still an open question.

We next provide a summary of applications of \EGBs\ in rings with infinite numbers of indeterminates.  

\subsubsection{Group theory and Chemistry.}

The first use of the concept ``finite up to symmetry" for polynomial rings that we are aware of is in the group theory work of Cohen in \cite{cohen1967laws}.  
Independently, it was problems in algebraic chemistry \cite{ruch1967vandermondesche}, brought to the attention of the authors of \cite{aschenbrenner2007finite} by Andreas Dress, that motivated further applications of asymptotic polynomial algebra in chemistry \cite{Draisma08b}.

\subsubsection{Toric Algebra and Algebraic Statistics.}

A major inspiration for asymptotic algebra arises from studying chains of toric ideals, many of which arise naturally in algebraic statistics.  The series of works \cite{Hillar13, hillar2016corrigendum, draisma2013noetherianity, KKL:equivariant-markov, Krone:egb-toric} have developed fundamental finiteness properties of these structures, but many questions remain open, as we outline in Section \ref{sec:challenges}.  

In this regard, one of the major motivations for \EGBs\ and infinite symbolic algebra are their application to the problem of sampling from conditional distributions by algebraic methods \cite{diaconis1998algebraic}.  At its essence, the strategy is to find a collection of elementary moves through model space that preserves the sufficient statistics of the data.   The idea then is to consider growing families of model classes and show that, up to obvious symmetries, only a finite set of moves suffices for all infinite numbers of models (e.g., \cite{aoki2003minimal, santos2003higher, hocsten2007finiteness, drton2007algebraic, Draisma08b, Brouwer09e, draisma2009ideals, hillar2012finite, draisma2015finiteness}).  Typically, these moves correspond to elements of a Gr\"obner basis or at least a generating set for some ideal.

\subsubsection{Invariants.}
Recently, Nagel and R\"omer~\cite{Nagel} have introduced Hilbert series for Noetherian infinite-dimensional rings. Their original theoretical treatment that leads to a proof of rationality of the Hilbert series, in principle, also leads to an effective procedure to compute the series. For an ideal generated up to symmetry by one monomial, such computation was carried out in~\cite{gunturkun2016equivariant}. An alternative approach of~\cite{krone2016hilbert} computes the Hilbert series given an equivariant Gr\"obner basis as the generating function counting words in a regular language.  

\subsection{Finiteness up to symmetry in general.}
Although the \EGBs\ described in this article may not directly apply, finiteness up to symmetry plays the central role in the following results (this list is by no means exhaustive).  

It appears in homological stability~\cite{randal2013homological, church2012homological},  the moduli space of $n$ points in a line~\cite{howard2009equations}, geometry as the positivity of the embedding line bundle grows \cite{ein2012asymptotic}, syzygies of Segre embeddings~\cite{snowden2013syzygies}, Betti tables as their length goes to infinity~\cite{ein2015asymptotics}, tensor geometry~\cite{draisma2014bounded, draisma2015finiteness}, and limiting Grassmannians~\cite{draisma2015plucker}.
Gr\"obner methods have also been used to understand representations of combinatorial categories \cite{sam2016grobner}.

\subsection{Goals and structure}

Since their introduction, Gr\"obner bases techniques have improved immensely.  We believe that in this new setting of infinite-dimensional polynomial algebras, which demands far more computational power, algorithmic development is at the beginning of a similar road, with similar advances ahead. Our aims here are to outline the current state of effective computation in this setting and to provide a background for researchers to start tackling problems in this exciting  domain. 

After some preliminaries in Section~\ref{prelim}, we quickly move on to describing equivariant Gr\"obner bases algorithms in Section~\ref{EGB}.  Section~\ref{sec:signature} goes on to explain a modern signature-based approach and a strategy inspired by it for an equivariant Buchberger's algorithm. The final Section~\ref{sec:challenges} outlines computational and theoretical challenges for future exploration.

\section{Preliminaries}\label{prelim}

Let $R$ be a commutative $K$-algebra equipped with a left action of monoid $\Pi$ (a $\Pi$-algebra structure).  We mainly consider the case where $R$ has the structure of a monoid algebra; that is, for some abelian monoid $\mon$, the elements of $R$ consist of formal sums of elements of $\mon$ with coefficients in the field $K$.  An example of such a monoid algebra is polynomial ring $R = K[X]$ with variables from the set $X$.  In this case, $\mon$ is the free abelian monoid generated by $X$, which we will denote by $[X]$.  To make our notation consistent with the polynomial case, we will denote the monoid algebra of $\mon$ over $K$ by $K\mon$ even though this is not standard (often it is written as $K[\mon]$, but this creates ambiguity with polynomial rings).  We also generally refer to elements of $\mon$ as ``monomials'' in analogy to the polynomial case.  Additionally, we will assume that $\Pi$ acts on monoid algebra $R$ through a $\Pi$-action on $\mon$ by monoid homomorphisms.

Our particular focus in this paper is when $\Pi$ is an infinite symmetric group $\SymN$ or certain related monoids.  For our purposes, we take $\SymN$ to be the group of all finite permutations of $\B N$ (i.e., permutations that fix all but a finite number of elements).

\begin{example}
 Let $R = K[x_1,x_2,x_3,\ldots]$ with $\SymN$ acting on $R$ by permuting the variables, so that $\sigma x_i = x_{\sigma(i)}$.
\end{example}

\begin{definition}
 An ideal $I \subseteq R$ is a {\em $\Pi$-invariant ideal} if $\sigma I \subseteq I$ for all $\sigma \in \Pi$.
\end{definition}

The ring $R$ is both an $R$-algebra and a $\Pi$-algebra, and there is a ring $R*\Pi$ which captures both of these actions, and which will be referred to as the {\em twisted monoid ring} of $\Pi$ with coefficients in $R$.  The elements of $R*\Pi$ are of the form $\sum_{\sigma \in \Pi} f_{\sigma}\cdot \sigma$ with each $f_\sigma \in R$ and only a finite number nonzero.  The additive structure is the same as the usual monoid ring, but multiplication is ``twisted'':
 \[ (f\cdot \sigma)(g \cdot \tau) = f\sigma(g) \cdot \sigma\tau,\]
where $\sigma(g)$ denotes the element of $R$ obtained by acting on $g$ by $\sigma$.

The ring $R$ is a $R*\Pi$-module, and the definition of $\Pi$-invariant ideals can be restated as the collection of $R*\Pi$-submodules of $R$.

When $R = K\mon$ with $\Pi$ acting on $\mon$, we can define a monoid $\mon *\Pi$ whose elements are pairs in $\mon \times \Pi$ with monoid operation:
 \[ (m, \sigma)(n, \tau) = (m\sigma(n), \sigma\tau). \]
There is a left action of $\mon*\Pi$ on $\mon$, and the elements of $\mon *\Pi$ are the ``monomials'' of $R*\Pi$.

\begin{definition}
 A $\Pi$-invariant ideal $I \subseteq R$ is {\em $\Pi$-finitely generated} if there is a finite set $F \subseteq I$ such that the $\Pi$-orbits of the elements of $F$ generate $I$.  The ring $R$ is called {\em $\Pi$-Noetherian} if every $\Pi$-invariant ideal in $R$ is $\Pi$-finitely generated.
\end{definition}
If a $\Pi$-invariant ideal $I$ is generated by the $\Pi$-orbits of a set $F$, we shall write:
 \[ I = \ideal{F}_{\Pi}. \]
Such a set $F$ generates $I$ as an $R*\Pi$-module.

We can also say that monoid $\mon$ with $\Pi$-action is $\Pi$-finitely generated if it is generated by the $\Pi$-orbits of a finite number of elements.  Then, $R = K\mon$ is $\Pi$-finitely generated as a $K$-algebra.

\begin{example}
Continuing the example of $R = K[x_1,x_2,x_3,\ldots]$ with $\SymN$ action, the ideal $\F m = \ideal{x_1,x_2,x_3,\ldots}$ is a $\SymN$-invariant ideal.  Moreover, it is $\SymN$-finitely generated because $\F m = \ideal{x_1}_{\SymN}$.  Also, $R$ is a $\SymN$-finitely generated $K$-algebra with generator $x_1$.
\end{example}

\begin{definition}
 Let $R$ be a $\SymN$-algebra.  For $f \in R$, the {\em width} of $f$ is the smallest integer $n$ such that for every $\sigma \in \SymN$ that fixes $\{1,\ldots,n\}$, $\sigma$ also fixes $f$.  The width of $f$ is denoted $w(f)$.  If no such integer $n$ exists, then $w(f) := \infty$.  For a set $F \subseteq R$, its width is $w(F) := \max_{f \in F}\{w(f)\}$.
\end{definition}
If every element of $R$ has finite width, we say that $R$ satisfies the {\em finite width condition}.  This is primarily the situation we want to address in this paper, and so we shall assume from here forward that all rings with $\Sym$-action satisfy the finite width condition unless stated otherwise.  For a $\SymN$-invariant ideal $I \subseteq R$ and an integer $n$, we can define the $n$th truncation of $I$ as:
 \[ I_n := \{ f \in I \mid w(f) \leq n \}. \]
The set $I_n$ is naturally a $\F S_n$-invariant ideal of $R_n$.  If $R$ satisfies the finite width condition, then $I$ is the union of all its truncations.  Moreover, if $I$ is $\SymN$-finitely generated, there is sufficiently large $n \in \B N$ such that $I = \ideal{I_n}_{\SymN}$.

The definition of width also applies to $\Pi = \Inc$, the monoid of strictly increasing functions, which is introduced below.
 
\begin{definition}
 Given $R = K\mon$ with $\Pi$ acting on $\mon$, there is a natural partial order $|_\Pi$ on $\mon$ called the {\em $\Pi$-divisibility partial order} defined by $a |_\Pi b$ if there exists $\sigma \in \Pi$ such that $\sigma a$ divides $b$.  Equivalently, $a |_\Pi b$ iff $b \in \ideal{a}_\Pi$.
\end{definition}

Recall that a monomial order on $R = K\mon$ is a total order $\leq$ on $\mon$ that is a well-order and that respects multiplication (i.e., if $a \leq b$ then $ac \leq bc$ for all $c \in \mon$).

\begin{definition}
 A monomial order $\leq$ on $R = K\mon$ is said to {\em respect $\Pi$} if whenever $a \leq b$, then $\sigma a \leq \sigma b$ for all $\sigma \in \Pi$.
\end{definition}

Therefore, order $\leq$ is a $\Pi$-respecting monomial order on $R$ if $\leq$ is a total well-order on $M$ that respects the action of $\mon*\Pi$.
We now have all the tools to describe the $\Pi$-equivariant version of Gr\"obner bases.
\begin{definition}
 Let $R = K\mon$ be a monoid ring with $\Pi$ action on $\mon$, and let $\leq$ be a $\Pi$-respecting monomial order.  Given a $\Pi$-invariant ideal $I \subseteq R$, a {\em $\Pi$-equivariant Gr\"obner basis} of $I$ is a set $G \subseteq I$ such that the $\Pi$ orbits of $G$ form a Gr\"obner basis of $I$:
 \[ \ideal{\LT \Pi G} = \LT I. \]
\end{definition}
We require $\leq$ to be a $\Pi$-respecting order because it is equivalent to the condition that:
\[ \LT \sigma f = \sigma \LT f,\]
for all $f \in R$ and $\sigma \in \Pi$.  Therefore, with such an order, we have:
 \[ \ideal{\LT G}_{\Pi} = \ideal{\LT \Pi G} = \LT I. \]
This also implies that $\LT I$ is a $\Pi$-invariant ideal.  Note that since $\Pi$ orbits of $G$ are a Gr\"obner basis of $I$, we naturally have $\ideal{G}_\Pi = I$.

\begin{proposition}[Remark 2.1 of \cite{Brouwer09e}]\label{prop:nogroup}
 Let $\Pi$ be a group which acts nontrivially on $\mon$.  Then $K\mon$ has no $\Pi$-respecting monomial orders.
\end{proposition}
\begin{proof}
 Suppose that $\leq$ is a $\Pi$-respecting order and choose $\sigma \in \Pi$ and $m \in \mon$ such that $m \neq \sigma m$.  If $m > \sigma m$, then $\sigma^n m > \sigma^{n+1} m$ for all $n$, and thus it follows that:
  \[ m > \sigma m > \sigma^2 m > \cdots \]
 is an infinite descending chain of monomials, contradicting the fact that $\leq$ is a well-order.  If $m < \sigma m$, then $m > \sigma^{-1} m > \sigma^{-2} m > \cdots$ is an infinite descending chain.
\end{proof}

In particular, this means that $R$ with nontrivial $\Sym$ action has no $\Sym$-respecting monomial orders.  To deal with this problem, a related monoid is introduced to replace $\Sym$ that allows for monomial orders but is somehow large enough compared to $\Sym$ not to break properties like finite generation.

Define the {\em monoid of strictly increasing functions} as:
\[ \Inc := \{ \rho:\B N \to \B N \mid \text{ for all } a < b, \rho(a) < \rho(b) \}. \]
For any $\Sym$-algebra $R$ with the finite width property, there is a natural action of $\Inc$ on $R$ as follows.  Fixing $f \in R$, for any $\sigma \in \Sym$ the value of $\sigma f$ depends only on the restriction $\sigma|_{[w(f)]}$ considering $\sigma$ as a function $\B N \to \B N$.  For any $\rho \in \Inc$, there exists $\sigma \in \Sym$ such that $\sigma|_{[w(f)]} = \rho|_{[w(f)]}$ and defines $\rho f = \sigma f$.  It can be checked that this gives a well-defined action of $\Inc$ on $R$.

It immediately follows from this definition that $\Inc f \subseteq \Sym f$.  Despite the fact that $\Inc$ is not a submonoid of $\Sym$, it behaves like one in terms of its action on $R$.  An injective map $\sigma|_{[w(f)]}: [w(f)] \to \B N$ can always be factored into $\rho' \circ \tau$ with $\tau \in \F S_{w(f)}$ and $\rho':[w(f)] \to \B N$ a strictly increasing function.  The map $\rho'$ can be extended to some $\rho \in \Inc$, and then $\sigma f = \rho (\tau f)$.  Thus, we have:
 \[ \Sym f = \bigcup_{\tau \in \F S_{w(f)}} \Inc(\tau f). \]
The fact that the $\Sym$-orbit of any $f$ is a finite union of $\Inc$-orbits implies the following statements.

\begin{proposition}
 Let $R$ be a $\Sym$-algebra satisfying the finite width condition, and let $I\subseteq R$ be a $\Sym$-invariant ideal.
 \begin{itemize}
  \item $I$ is $\Inc$-invariant.
  \item $I$ is $\Sym$-finitely generated if and only if $I$ is $\Inc$-finitely generated.
  \item If $R$ is $\Inc$-Noetherian then $R$ is $\Sym$-Noetherian.
 \end{itemize}
\end{proposition}

\begin{remark}\label{rem:IncT}
For practical purposes, we may replace $\Inc$ with $\IncT$, the monoid of all increasing maps
$\pi:\mathbb N \to \mathbb N$ such that $\im(\pi)$ has a finite complement. 
For $i \in \mathbb N$, let $\tau_i$ denote the element of $\Pi$ defined by
\[ \tau_i(j)=
\begin{cases} 
	j & \text{if $j<i$, and}\\
	j+1 & \text{if $j \geq i$.}
\end{cases}
\]
The maps $\tau_i$ generate $\Pi$, and they satisfy relations
\[ \tau_{j+1}\tau_i=\tau_i \tau_j \text{ if }j \geq i. \]
This gives a presentation of $\Pi$, and any element of $\Pi$
has a unique expression of the form $\tau_{i_1} \cdots \tau_{i_d}$
with $i_1 \leq \ldots \leq i_d$.
\end{remark}

When computing Gr\"obner bases of $\Sym$-invariant ideals, we will work with the $\Inc$ action instead.  If $G$ is an $\Inc$-equivariant Gr\"obner basis for $\Sym$-invariant ideal $I$, then the $\Sym$-orbits of $G$ also form a Gr\"obner basis of $I$.  Generally, the rings we are interested in will have $\Inc$-respecting monomial orders.

\begin{example}
 Let $R = K[x_1,x_2,\ldots]$ with $\Inc$-action defined by $\rho \cdot x_i = x_{\rho(i)}$.  The lexicographic order $\leq$ on the monomials of $R$ with $x_1 < x_2 < x_3 < \cdots$ is a $\Inc$-respecting monomial order.  This is the only possible lexicographic order on $R$ that respects $\Inc$.  There are also a graded lexicographic and a graded reverse lexicographic order on $R$ that respect $\Inc$.  There is no $\Inc$-respecting monomial order on $R$ that is defined by a single weight vector in $\B R^{\B N}$.
\end{example}

It is an open question to characterize all possible $\Inc$-respecting monomial orders on a given ring $K\mon$ with $\Inc$ action.  We can make the following statement about such orders.

\begin{proposition}
 If $\leq$ is a $\Pi$-respecting monomial order on $K\mon$, then $\leq$ refines the $\Pi$-divisibility quasi-order $|_\Pi$.
\end{proposition}
\begin{proof}
 Suppose $a$ and $b$ are monomials with $a |_\Pi b$, so there is some pair $\sigma \in \Pi$, $c \in \mon$ such that $c\sigma a = b$.  From the proof of Proposition \ref{prop:nogroup}, we see that $a \leq \sigma a$.  Since $1 \leq c$ and $\leq$ respects multiplication, it follows that $\sigma a \leq c\sigma a = b$.
\end{proof}

One implication of this proposition is that if $K\mon$ has a $\Pi$-respecting monomial order then the $\Pi$-divisibility quasi-order must be a partial order (i.e., it has the \textit{anti-symmetry} property: if $a \geq b$ and $a \leq b$ then $a = b$).  If anti-symmetry fails for $|_\Pi$, it will also fail for any refinement.

If $R$ is $\Pi$-Noetherian with a $\Pi$-respecting monomial order, then any $\Pi$-invariant ideal $I \subseteq R$ will have a finite $\Pi$-equivariant Gr\"obner basis.  This follows from the fact that $\LT I$ is $\Pi$-finitely generated.  We recount two previous results that give examples of $\Inc$-Noetherian rings, and they will be directly relevant to the results of this paper.

\begin{theorem}[Theorem 1.1 of \cite{hillar2012finite}]\label{thm:HS}
 Let $X = \{x_{ij} \mid i \in [k], j \in \B N\}$, and let $\Sym$ act on $[X]$ by permuting the second index: $\sigma x_{ij} = x_{i\sigma(j)}$ for $\sigma \in \Sym$.  Then, $K[X]$ is $\Inc$-Noetherian.
\end{theorem}

\begin{theorem}[Theorem 1.1 of \cite{draisma2013noetherianity}]\label{thm:DEKL}
 Let $K[Y]$ be a $\Sym$-algebra with $\Sym$ action on variable set $Y$.  Suppose $Y$ has a finite number of $\Sym$-orbits, and $K[Y]$ satisfies the finite width condition.  For $K[X]$ defined as in Theorem \ref{thm:HS}, let $\phi$ be a monomial map:
  \[ \phi: K[Y] \to K[X]. \]
 Then, the following hold:
 \begin{itemize}
  \item $\ker \phi$ is $\Inc$-finitely generated,
  \item $\im \phi$ is $\Inc$-Noetherian.
 \end{itemize}
\end{theorem}

The conditions on the ring $K[Y]$ in Theorem \ref{thm:DEKL} are quite general although \cite{hillar2012finite} proves that such rings are generally not $\Sym$-Noetherian.  They give the example of $K[Y]$ where $Y = \{y_{ij} \mid i, j \in \B N\}$ with $\sigma y_{ij} = y_{\sigma(i)\sigma(j)}$ for $\sigma \in \Sym$ and prove that Noetherianity fails.

When $R$ is not $\Pi$-Noetherian, we do not know in general if a $\Pi$-finitely generated ideal $I\subseteq R$ has a finite $\Pi$-equivariant Gr\"obner basis, or if so, for which monomial orders.  However,~\cite{Krone:egb-toric} shows that the $\Sym$-invariant toric ideal $\ker \phi$ as in Theorem \ref{thm:DEKL} does have finite $\Inc$-equivariant Gr\"obner bases for specifically chosen monomial orders.  This allows for an algorithm to compute a Gr\"obner basis of $\ker \phi$ given $\phi$.

\section{Equivariant Buchberger algorithm}\label{EGB}

\subsection{Description of the algorithm}
First proposed in \cite{aschenbrenner2007finite} and formalized in \cite{Brouwer09e}, the classical Buchberger's algorithm~\cite{buchberger1965algorithmus} may be adapted to the equivariant setting in a straightforward way.

Let $R = K\mon$ with $\Pi$ acting on $\mon$, and let $\leq$ be a $\Pi$-respecting monomial order.  For $f,g \in R$, we say that $g$ {\em $\Pi$-reduces} $f$ if $\LT g |_{\Pi} \LT f$ and the reduction is $f - \frac{\LC(f)}{\LC(g)}mg$ where $m \in \mon *\Pi$ is such that $\LT f = m\LT g$ (and $\LC(f)$ denotes the lead coefficient of $f$).  For $G \subseteq R$, a {\em $\Pi$-normal form} of $f$ with respect to $G$, denoted $\NF_{\Pi G}(f)$, is the result of repeated $\Pi$-reductions of $f$ by elements of $G$ until no more reductions are possible.  Equivalently, $\NF_{\Pi G}(f)$ is a normal form of $f$ with respect to $\Pi G$.

The equivariant Buchberger's algorithm is described below, which departs from the conventional Buchberger's algorithm only at the step of adding new S-pairs to the list $S$.  The necessity and extent of this departure becomes clear with the definition of $O_{f,g}$ and the {\em finite S-pair condition} (Definition~\ref{def:finite-s-pair}) that are given after the description of the algorithm.

\begin{algorithm}[Brouwer--Draisma \cite{Brouwer09e}]\label{alg:Buchberger}
$G = \alg{Buchberger}(F)$
\begin{algorithmic}[1]
\REQUIRE $F$ is a finite set of elements in $R = K\mon$ with $\Pi$ acting on $\mon$ and satisfying the finite S-pair condition.
\ENSURE $G$ is $\Pi$-equivariant Gr\"obner basis of $\ideal{F}_{\Pi}$.

\smallskip \hrule \smallskip

\STATE $G\gets F$
\STATE $S\gets \bigcup_{f,g\in G} O_{f,g}$
\WHILE{$S\neq\emptyset$}
	\STATE pick $(h_1,h_2) \in S$
	\STATE $S\gets S\setminus\{(h_1,h_2)\}$ 
	\STATE $h \gets \NF_{\Pi G}(h_1 - \frac{\LC(h_1)}{\LC(h_2)}h_2)$
  	\IF{$h \neq 0$}
		\STATE $G\gets G\cup \{h\}$
		\STATE $S\gets S\cup \left(\bigcup_{g\in G}O_{g,h}\right)$
	\ENDIF
\ENDWHILE
\smallskip \hrule \smallskip
\end{algorithmic}
\end{algorithm}

Given $f,g \in R$ define:
 \[ \C S_{f,g} := \{(m_1f,m_2g) \mid m_1,m_2 \in \mon * \Pi \text{ such that } \LT m_1f = \LT m_2g\}. \]
This collection is closed under the diagonal action of $\mon *\Pi$, making $\C S_{f,g}$ a $\mon *\Pi$-module.  

\begin{definition}\label{def:buchberger-criterion}
A set $G \subseteq R$ satisfies the {\em equivariant Buchberger criterion} if for all $(h_1,h_2) \in \bigcup_{f,g\in G} \C S_{f,g}$:
 \[ \NF_{\Pi G}(h_1 - \tfrac{\LC(h_1)}{\LC(h_2)}h_2) = 0. \]
\end{definition}

The set $G$ is a $\Pi$-equivariant Gr\"obner basis of $\ideal{G}_{\Pi}$ if and only if it satisfies the equivariant Buchberger criterion.  The proof of this fact follows by applying the usual Buchberger criterion to the set $\Pi G$ (see Theorem 2.5 of \cite{Brouwer09e}).

For each pair $f,g \in G$, we need not check the criterion on every pair in the infinite set $\C S_{f,g}$.  It is instead sufficient to check on a $\mon * \Pi$ generating set of $\C S_{f,g}$, which we denote $O_{f,g}$.  Still, in general, it may be that no finite generating set of $\C S_{f,g}$ exists, in which case we cannot apply the algorithm in finite time.

\begin{definition}\label{def:finite-s-pair}
 A $\Pi$-algebra $R = K\mon$ has the {\em finite S-pair condition} if for any $f,g \in R$, the set $\C S_{f,g}$ is finitely generated as a $\mon * \Pi$-module.  In \cite{Brouwer09e}, this condition is referred to as ``EGB4.''
\end{definition}

When $\Pi$ is trivial and $R$ is a polynomial ring (the setting of the conventional Buchberger's algorithm), $\C S_{f,g}$ is generated by a single pair $(m_1 f,\; m_2 g)$ where:
$$
m_1 = \lcm(\LT f, \LT g)/\LT(f),\quad
m_2 = \lcm(\LT f, \LT g)/\LT(g)
\,.
$$  
This generator is typically referred to as the {\em S-pair} of $f,g$.  Therefore, $R$ in this case satisfies the finite S-pair condition, and the equivariant Buchberger's algorithm specializes to the conventional Buchberger's algorithm.

\begin{proposition}
 If $R$ is a polynomial ring $R = K[Y]$ with $\Inc$-action on $[Y]$ satisfying the finite width condition, then $R$ has the finite S-pair condition.
\end{proposition}
\begin{proof}
 Fix $f,g \in R$.
 Since $R$ is a polynomial ring, for fixed $\sigma_1,\sigma_2 \in \Inc$, all elements of $\C S_{f,g}$ of the form $(m_1\sigma_1 f, m_2\sigma_2 g)$ with $m_1,m_2 \in \mon$ are monomial multiplies of the usual S-pair of $\sigma_1 f, \sigma_2 g$:
  \[ \left(\frac{m}{\LT \sigma_1 f} \sigma_1 f,\; \frac{m}{\LT \sigma_2 g} \sigma_2 g\right),\]
 where $m = \lcm(\LT \sigma_1 f,\LT \sigma_2 g)$.

 Any $f,g \in R$ have finite width so that $\sigma_1 f$ depends only on $\sigma_1|_{[w(f)]}$, and similarly for $\sigma_2 g$.  In fact, we can always factor the pair as:
  \[ (\sigma_1 f, \sigma_2 g) = \rho(\sigma'_1 f, \sigma'_2 g),\]
 for some $\rho \in \Inc$, while $\sigma'_1:[w(f)] \to [w(f) + w(g)]$ and $\sigma'_2:[w(g)] \to [w(f) + w(g)]$ are strictly increasing functions.  Here $\sigma'_1$ and $\sigma'_2$ are chosen to ``interlace'' the variables of $f$ and $g$ in the same way as $\sigma_1,\sigma_2$.  (To consider $\sigma'_1,\sigma'_2$ as elements of $\Inc$, take any choice of extensions to maps on $\B N$.)
 
 Then, $\C S_{f,g}$ is generated by the finite set of pairs of the form:
  \[ \left(\frac{m}{\LT \sigma'_1 f} \sigma'_1 f,\; \frac{m}{\LT \sigma'_2 g} \sigma'_2 g\right),\]
 with $\sigma'_1:[w(f)] \to [w(f) + w(g)]$ and $\sigma'_2:[w(g)] \to [w(f) + w(g)]$ where $m = \lcm(\LT \sigma'_1 f,\LT \sigma'_2 g)$.
\end{proof}
\begin{figure}[ht]\label{fig:interlace}
  \centering
  \includegraphics[width=.9\columnwidth]{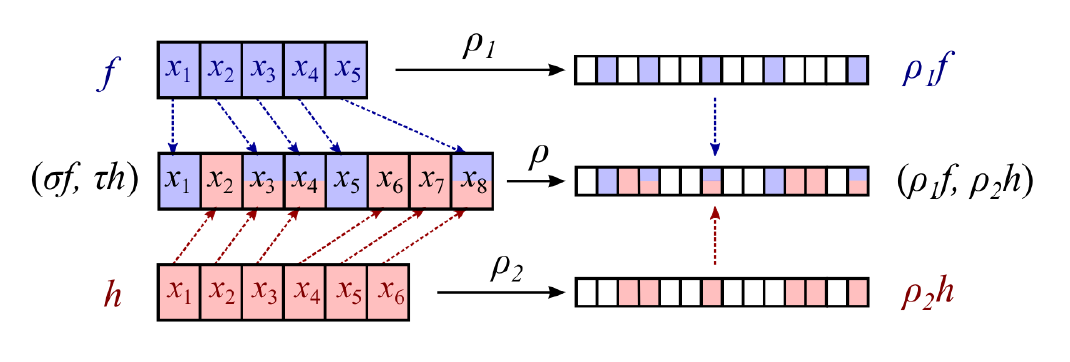}
  \caption{For $f$ and $g$ of width 5 and 6 respectively, any S-pair, $(\rho_1 f,\rho_2 g)$, is in the orbit of some S-pair obtained from an ``interlacing'' of $[5]$ and $[6]$, $(\sigma f,\tau g)$.}
\end{figure}

We note that Algorithm \ref{alg:Buchberger} is guaranteed to terminate when $R$ is $\Pi$-Noetherian.  Let $G_0,G_1,\ldots$ be the value of $G$ at each step.  The initial ideals of these sets form a strictly increasing chain of $\Pi$-invariant monomial ideals:
 \[ \ideal{\LT G_0}_\Pi \subsetneq \ideal{\LT G_1}_\Pi \subsetneq \cdots,\]
which must terminate.  However, without Noetherianity, we offer no termination guarantee of the algorithm as stated above, even when a finite equivariant Gr\"obner basis for the ideal exists.  Algorithm~\ref{alg:truncBuch} is a modification of the algorithm which repairs this when $\Pi = \Inc$, a finite equivariant Gr\"obner basis exists, and the truncated rings $R_n$ are Noetherian.

\subsection{Termination of $\Inc$-equivariant Buchberger}
Let $R = K\mon$ with $\Inc$ action on $\mon$, with $R$ satisfying the finite width and finite S-pair conditions, and with each truncation $R_n$ a Noetherian ring.  Let $I \subseteq R$ be a $\Inc$-invariant ideal which is $\Inc$-generated by finite set $F$, and, moreover, has finite $\Inc$-equivariant Gr\"obner basis $G$.  Define the {\em generator truncation} of $I$ to be $\tilde{I}_{F,n} := \ideal{\Inc F \cap R_n} \cap R_n$.  Note that $\tilde{I}_{F,n} \subseteq I_n$, but, in general, equality does not hold.  For $f \in I$, define $w_F(f)$ to be the minimum value of $n$ for which $f \in \tilde{I}_{F,n}$.

The truncated EGB algorithm takes a finite generating set $F$ as its input.  For each successive $n \geq w(F)$, it computes a set $G_n$ such that $\Inc G_n \cap R_n$ is a Gr\"obner basis for $\tilde{I}_{F,n}$.  Then it checks if $G_n$ is a $\Inc$-equivariant Gr\"obner basis of $I$ using the equivariant Buchberger criterion (Definition~\ref{def:buchberger-criterion}), and if so returns $G_n$.

\begin{algorithm}\label{alg:truncBuch}
$G = \alg{TruncatedEGB}(F)$
\begin{algorithmic}[1]
\REQUIRE $F$ is a finite set of elements in $R = K\mon$ with $\Inc$ acting on $\mon$, $R$ satisfies the finite width and finite S-pair conditions, and each $R_n$ is Noetherian.
\ENSURE $G$ is a $\Inc$-equivariant Gr\"obner basis of $I := \ideal{F}_{\Inc}$.

\smallskip \hrule \smallskip

\STATE $G\gets F$
\STATE $n\gets w(F)$
\WHILE{$G$ not a $\Inc$-equivariant Gr\"obner basis of $I$}
	\STATE $G\gets$ Gr\"obner basis of $\tilde{I}_{F,n}$
	\STATE $n \gets n+1$
\ENDWHILE
\smallskip \hrule \smallskip
\end{algorithmic}
\end{algorithm}

\begin{proof}[Proof of termination (supposing a finite EGB for $I$)]
For each $n$, let $G_n$ denote the value of $G$ after that step.  Computing $G_n$ is a finite process since it takes place in $R_n$, which is Noetherian.  $G_n$ is a finite set, and so it has a finite number of S-pairs to be checked.  Therefore, testing whether $G_n$ is a $\Inc$-equivariant Gr\"obner basis is finite.

It remains to be proved that $G_n$ is a $\Inc$-equivariant Gr\"obner basis for some value of $n$.  If $H$ is a $\Inc$-equivariant Gr\"obner basis of $I$, for any $h \in H$ we have $h \in \tilde{I}_{F,n}$ for all $n \geq w_F(h)$, so $\LT(h) \in \LT(\tilde{I}_{F,n})$.  Thus, there is some $g \in G_n$ with $\LT(g)|_{\Inc} \LT(h)$.  For $n = \max_{h\in H} w_F(h)$, the initial ideal $\ideal{\LT(G_n)}_{\Inc}$ then necessarily contains $\ideal{\LT(H)}_{\Inc}$, and so $G_n$ is a $\Inc$-equivariant Gr\"obner basis of $I$.
\end{proof}

In practice, $G_n$ can be computed either using a traditional Gr\"obner basis algorithm on input $\Inc F \cap R_n$ or using an equivariant Buchberger's algorithm on input $F$ with the following two caveats:
\begin{itemize}
 \item Consider only S-pairs $(m_1f, m_2g)$ with $m_1f$ and $m_2g$ both having width $\leq n$,
 \item Perform only reductions such that the outcome has width $\leq n$.
\end{itemize}
Moreover, we do not need to restart the algorithm from scratch for each $n$: $G_{n-1} \cup F$ can be used as the input for the $n$th step instead of $F$.

Suppose $R$ has the form $K[Y]$ and each $R_n = K[Y_n]$ for some $Y_n \subseteq Y$.  If $\leq$ is a width order (a monomial order such that $w(a) < w(b)$ implies $a < b$), the second condition is satisfied automatically since reductions cannot increase the width.  Therefore, the normal form of a given S-pair does not depend on $n$ and only needs to be computed once.  As a result, we can use Algorithm \ref{alg:Buchberger}, queuing S-pairs by width so that the smallest width S-pairs are considered first.  The algorithm terminates once the queue is empty.  A separate check for whether $G_n$ is a $\Inc$-equivariant Gr\"obner basis for $I$ is not needed since this is equivalent to reducing all S-pairs in the queue.

\subsection{Macaulay2 package}

We have implemented several strategies for computing equivariant Gr\"obner bases in a package: \begin{center}
\verb|EquivariantGB|~(see {\tt http://rckr.one/EquivariantGB.html}) 
\end{center}
for Macaulay2~\cite{grayson2002macaulay}, a software system for computational algebraic geometry and commutative algebra.

The main command of our Macaulay2 package, \verb|egb|, has an optional argument that determines how the computation is done:
\begin{itemize}
\item 
\verb|egb(...,Algorithm=>Buchberger)| uses Algorithm~\ref{alg:Buchberger},
\item 
\verb|egb(...,Algorithm=>Incremental)| uses Algorithm~\ref{alg:truncBuch},
\item 
\verb|egb(...,Algorithm=>Signature)| uses the approach in~\S\ref{sec:signature}.
\end{itemize}

\begin{remark}\label{rem:4ti2}
With the assumptions of Theorem~\ref{thm:DEKL}, one can operate  with truncated toric ideals and use specialized lattice based Gr\"obner bases algorithms to improve performance. 
Our Macaulay2 command for that, \verb|egbToric|, outsources heavy computation to \verb|4ti2|~(see \cite{4ti2}),  a special software package for algebraic, geometric, and combinatorial problems on linear spaces.   
\end{remark}

\section{A signature-based approach}\label{sec:signature}

In this section, we describe an approach to computing equivariant \GBs\ that utilizes the information stored in {\em signatures}. 

Signature-based algorithms for computing \GBs\ in the most common (finite-dimensional, commutative) setting acquired popularity due to Faugere's F5 (see a short description in \S 4 of Chapter 10 of the new edition of Cox, Little, and O'Shea~\cite{Cox-Little-OShea:I-V-A}).      
We give a description of one of the signature-based approaches due to Gao {\em et al.} in \cite{Gao-Volny-Wang:signature-GBs}, followed by its modification needed to compute equivariant \GBs.  

\subsection{Strong \GB}

Let $I=\ideal{F} \subset R=K[x_1,\ldots,x_n]$, where $|F|=r\in\mathbb N$.

A subset $G$ of:
\[
P = \{(s,f)\in R^r\times R \mid f=s\cdot F = \sum_{i=1}^r s_iF_i\}
\]
is called a {\em strong \GB } if every nonzero pair is {\em top-reducible} by some pair in $G$.

A pair $(s_f,f)$ is {\em top-reducible} by $(s_g,g)$ if $\LM g |\LM f$ and 
for some $a$ with $\LM f = a \LM g$, we have $a\LM s_g\leq \LM s_f$.  If the reduction:
\[
(s_{f'},f') := (s_f-as_g,f-ag)
\]
has $\LM s_{f'} = \LM s_f$, then it is {\em regular top-reducible}.

If $G$ is a strong \GB, then by Proposition 2.2 of~\cite{Gao-Volny-Wang:signature-GBs}:
\begin{enumerate}
   \item $\{f \mid (s,f)\in G\}$ is a \GB\ of $I$, and 
   \item $\{s \mid (s,0)\in G\}$ is a \GB\ of the module of syzygies $\Syz(F) \subset R^r$.
\end{enumerate}
   
Take two pairs $p_f=(s_f,f)$ and $p_g=(s_g,g)$. For monomials $a$ and $b$ such that $a\LM f = b\LM g \in \lcm(\LM f,\LM g)$, form a {\em J-pair} by taking the ``larger side'' of the corresponding S-polynomial; e.g., if $a\LM s_f \geq b\LM s_g$, then the J-pair is $(as_f,af)$.

We denote the set of all J-pairs of $p_f$ and $p_g$ as $J_{p_f,p_g}$. Note that $\lcm(\LM f,\LM g)$, the {\em set} of lowest common multiples, has one element in our current setting as does $J_{p_f,p_g}$.  

\begin{example} If we have:
\begin{align*}
p_f &= (e_1+\cdots,\,x_1^2x_2+\cdots), \ p_g = (x_2e_1+\cdots,\,x_1x_2^2+\cdots),
\end{align*}
then, since $x_2\LM s_f < x_1\LM s_g$, we have:
\[
J_{p_g,p_f} = \{x_1p_g\} = \{(x_1x_2e_1+\cdots, x_1^2x_2^2+\cdots)\}.
\]
\end{example}

A pair $(s_f,f)$ is {\em covered} by $(s_g,g)$ if $\LM g | \LM f$ and for some $a$ such that $\LM f = a \LM g$, we have $a\LM s_g < \LM s_f$. 

\begin{algorithm} \label{alg:StrongBuchberger} 
$\alg{StrongBuchberger}(F)$

\begin{algorithmic}[1]
\REQUIRE $F \subset R$.
\ENSURE $G \cup S$ is a strong \GB\ for $F$.
\smallskip \hrule \smallskip

\STATE $G\gets \emptyset$, $S\gets \emptyset$ 
\STATE $J\gets \{(e_i,F_i):i\in r=|F|\} \subset R^r\times R$ 
\WHILE{$J\neq\emptyset$}
	\STATE pick $p_f = (s_f,f) \in J$; $J\gets J\setminus\{p_f\}$
	\STATE $p_h=(s_h,h) \gets$ {\em regular top-reduction} of $(s_f,f)$ with respect to $G$
  	\IF{$h \neq 0$}
		\STATE $G\gets G\cup \{p_h\}$
		\STATE append to $J$ all J-pairs $\bigcup_{(p_g)\in G}J_{p_g,p_h}$ not {\em covered} by $G \cup S$ 
        \ELSE 
                \STATE $S\gets S\cup\{(s_h,0)\}$
	\ENDIF
\ENDWHILE
\smallskip \hrule \smallskip
\end{algorithmic}
\end{algorithm}

Proof of termination relies on Noetherianity of the free module $R^r$. 

\subsection{Translation to an equivariant setting}
Let us return to an infinite-dimensional polynomial ring $R=K[X]$ with some $\Pi$-action. As a running example, take $R=K[x_i,\, i\in\mathbb N]$ with a $\Pi$-compatible order, $\Pi=\Inc$.

To draw parallels with the approach of the previous section, we need to work with pairs:
\[
P = \{(s,f)\in (R*\Pi)^r\times R \mid f=s\cdot F\}.
\]
Recall that, for instance, in our running example, one can think of $\Pi=\IncT = [\tau_i|i\in\mathbb N]\subset \Inc$ in Remark~\ref{rem:IncT}, so: \[R*\Pi = K[X]*\Pi = K([X]*\Pi).\] 
The semidirect product $[X]*\Pi$ is a {\em non-Noetherian noncommutative} monoid where every element can be written in a \emph{left standard form}: 
$$
x_{i_1}\cdots x_{i_c} \cdot \tau_{j_1} \cdots \tau_{j_d},
$$
with $i_1 \leq \ldots \leq i_c$ and $j_1 \leq \ldots \leq j_d$. 

Since the $\Pi$-divisibility order on $[X]*\Pi$ is not a well-partial-order (indeed, $\tau_i$ are pairwise not comparable), the (left) free module $(R*\Pi)^r$, $r\in\mathbb N$, is not Noetherian.  
In the presence of $F$ (the vector of generators of the given ideal) and with a fixed order on $R$, we define the \emph{Schreyer order} on $(R*\Pi)^r$ as follows. For two terms $me_i$ and $m'e_j$, with $m,m'\in [X]*\Pi$:
\begin{itemize}
\item compare $m\LM f_i$ and $m'\LM f_j$ using the order on $R$,
\item then break the ties according to the position (i.e., compare $i$ and $j$).
\end{itemize}
While we see the Schreyer order as natural in some sense, any term order compatible with the order on $R$ may be used. %Anton: need more time to see why the current EquivariantGB order works better.

A strong equivariant \GB, which can be defined similarly to a strong \GB\ in the previous section, is infinite (for a nonzero $\Pi$-invariant ideal). 
For instance, $I = R = K[x_1,x_2,\ldots]$ has a \GB\ $\{1\}$.
However, a strong \GB\ has to include the elements $\{(\tau_i-1)e_1 \mid i\in\mathbb N \} \subset (R*\Pi)^1$. 

We found a way to modify Algorithm~\ref{alg:StrongBuchberger} to compute an equivariant \GB.  It, of course, falls short of computing a strong equivariant \GB, but the partial information computed about the syzygies and the mechanism of top-reduction of J-pairs eliminate a large number of unnecessary iterations in a na\"ive implementation of an equivariant Buchberger's algorithm~(Algorithm \ref{alg:Buchberger}).

\begin{algorithm}\label{alg:egb-signature}
$\alg{EquivariantSignatureBuchberger}(F)$

\begin{algorithmic}[1]
\REQUIRE $F\subset R$ .
\ENSURE $G$ such that $\pi_2(G)$ is an equivariant \GB\ of $\ideal{F}_\Pi$.
\smallskip \hrule \smallskip
\STATE $r\gets |F|$.
\STATE $G\gets \emptyset$, $S\gets \emptyset$ 
\STATE $J\gets s\{(e_i,F_i):i\in r=|F|\} \subset R^r\times R$ 
\WHILE{$J\neq\emptyset$}
	\STATE pick $p_f = (s_f,f) \in J$; $J\gets J\setminus\{p_f\}$
	\STATE $p_h=(s_h,h) \gets$ {\em regular top-reduction} of $(s_f,f)$ with respect to $G$
  	\IF{$h \neq 0$}
		\STATE \hl{$h' \gets\NF_{\Pi\pi_2(G)} h$}
	        \IF{$h' \neq 0$}
			\IF{\hl{$h' \neq h$}}
				\STATE \hl{$r\gets r+1$}
				\STATE \hl{$p_h \gets (e_r,h')$}
			\ENDIF
			\STATE $G\gets G\cup \{p_h\}$
			\STATE append to $J$ all J-pairs $\bigcup_{(p_g)\in G}J_{p_g,p_h}$ not {\em covered} by $G \cup S$ 
		\ENDIF
        \ELSE 
                \STATE $S\gets S\cup\{(s_h,0)\}$
	\ENDIF
\ENDWHILE
\smallskip \hrule \smallskip
\end{algorithmic}
\end{algorithm}

The highlighted part of the algorithm ensures that it terminates for an input for which Algorithm \ref{alg:Buchberger} terminates.
Note that the rank $r$ (recall: $G$ and $S$ are contained in $(R*\Pi)^r\times R$) may grow as the algorithm progresses. 

\begin{example}
Consider the ideal $I = \ideal{F}_{\IncT}$ in the ring $R=K[x_i, y_{ij}\mid i,j\in\mathbb N, i>j]$ where $F = y_{21}-x_2x_1$.

The implementation of Algorithm~\ref{alg:egb-signature} produces the following output:
\begin{M2}
\begin{verbatim}
i1:  needsPackage "EquivariantGB";

i2 : -- QQ[x_0,x_1,...; y_(0,1),y(1,0),...] 
     -- (NOTE: indices start with 0, not 1)
     R = buildERing({symbol x, symbol y}, {1,2}, QQ, 2, 
                    MonomialOrder=>Lex, Degrees=>{1,2});

i3:  egbSignature(y_(1,0) - x_0*x_1)

...

...

-- 95th syzygy: (0, y_(6,0)*y_(4,3)*y_(2,1)*{2, 5, 6, 7, 8}*[0])

...

...

-- TOTAL covered pairs = 1528
                     
o3 = {- x x  + y   , ...  ...  ...
         1 0    1,0  

      - y   y    + y   y   , - y   y    + y   y   }
         3,2 1,0    3,1 2,0     3,1 2,0    3,0 2,1
\end{verbatim}
\end{M2}  

In particular, this computation shows that the kernel of the monomial map induced by $y_{ij}\mapsto x_ix_j$ is 
$\ideal{y_{43}y_{21}-y_{42}y_{31},y_{42}y_{31}-y_{41}y_{32}}_{\IncT}$.

The number of times a polynomial corresponding to a J-pair in the queue $J$ was reduced to zero is {\bf 95}. However, in this signature-based algorithm, the knowledge of 95 syzygies is still useful as their signatures are stored and may ``cover'' some J-pairs in the queue. The total number of covered J-pairs, 1528, could be taken as a measure of how many useless reductions are avoided.

There is an optional parameter: \begin{center} \verb|egbSignature(...,PrincipalSyzygies=>true)|,\end{center} that instructs the algorithm to construct the so-called \emph{principal syzygies}, the syzygies that correspond to the trivial commutation relations on the generators: $(\sigma F_i) (\sigma' F_j)-(\sigma'F_j)(\sigma F_i)=0$, $i\neq j$, where $\sigma,\sigma'\in\Pi$ are extensions of the maps $[w(F_i)] \to [w(F_i) + w(F_j)]$ and $ [w(F_j)] \to [w(F_i) + w(F_j)]$.
With this option, the previous computation produces a much larger number of syzygies, {\bf 1114}; however, there is no improvement obtained in terms of covered J-pairs, and the improvement in the number of J-pairs that need to be stored is insignificant. 

It is our understanding that in the usual setting (where the results of~\cite{Gao-Volny-Wang:signature-GBs} apply in their entirety), the introduction of principal syzygies leads to a significant speedup. While we can find examples where the effect of principal syzygies is nontrivial, it still seems to be negligible in the setting of this paper.
\end{example}

Our general conclusion at the moment of writing is that signature-based approaches are applicable for computing EGBs, however, the savings produced by eliminating unnecessary reductions are largely offset by the amount of J-pairs needed to be stored.
Perhaps with a more careful implementation of what we have proposed and some new ideas, one could overcome the bottlenecks of the required space complexity and the complexity of looking up  J-pairs. 

At the moment, implementations of algorithms that fall back onto highly optimized \GBs\ routines in the finite-dimensional setting (such as Algorithm~\ref{alg:truncBuch} and its variation in Remark~\ref{rem:4ti2}) seem to be the best practical choice. 

\section{Open questions and challenges}\label{sec:challenges}

In this final section, we raise several computational challenges and theoretical problems arising from equivariant Gr\"obner bases and asymptotic symbolic algebra.  Often these challenge problems can serve as benchmark tests for sharpening the methods of practitioners who are improving and implementing these new classes of algorithms.

\begin{problem}[Chains induced by a monomial]
Compute symbolically an EGB for the chain of toric ideals $I_n = \ker(y_{ij} \mapsto x_i^a x_j^b)$, $1 \leq i \neq j \leq n$, for small $a>b$ with $\gcd(a,b)=1$. (Compare to \cite{Hillar13, hillar2016corrigendum, KKL:equivariant-markov, draisma2013noetherianity, Krone:egb-toric}).
\end{problem}

The case $a=2,b=1$ is the only one explicitly computed (Theorem~\ref{monomthm}). A variant of this problem has the same statement apart from considering a smaller subset of indices: $1 \leq j < i \leq n$ (see \cite[Remark 6.3]{draisma2013noetherianity}) and also more indices:

\begin{problem}
Develop combinatorial methodology to understand kernels with more than two indices such as $\ker(y_{ijk} \mapsto x_i^3 x_j^2 x_k)$?  
\end{problem}

There are also some basic questions in the theory of EGBs that remain open.  For instance, it is not so well understood exactly which classes of ideals have finite generation, much less an equivariant Gr\"obner basis. 
While \cite{draisma2013noetherianity} gives a definitive answer for a large class of invariant toric ideals (i.e., the kernels of equivariant monomial maps), the following question is open.
\begin{question}[Kernel of a polynomial map]
Is there a finite set of generators (up to symmetry) for the chain $I_n = \ker(y_{ij} \mapsto f(x_i,x_j))$, $1 \leq i \neq j \leq n$, for a given polynomial $f \in \mathbb C[s,t]$?  
\end{question}

Even when finite generation is known, other problems still remain open.  
 We know that the kernels of monomial maps stabilize \cite{aschenbrenner2007finite, KKL:equivariant-markov, draisma2013noetherianity} and have EGBs with respect to certain monomial orders~\cite{Krone:egb-toric}. What if the monomial order is not particularly nice? What if the map is a general equivariant polynomial map?

\begin{question}
If the answer to the previous question is positive, is there a finite \EGB\ with respect to an arbitrary order?
\end{question}

One largely unexplored aspect of research efforts to date is the structure of term orders for equivariant Gr\"obner bases.  In the classical application of Gr\"obner bases, term orders play a significant role and such concepts as the \emph{Gr\"obner fan} and techniques such as the \emph{Gr\"obner walk} arise. These seem not to have equivalents in the equivariant setting in view of the following question.  

\begin{question} For $R = \mathbb C[x_0,x_1,x_2,\ldots]$, there are several natural monomial orders respecting $\Inc$-action and refining the $\Inc$-divisibility partial order: namely, lexicographic and graded lexicographic orders.  Are there any others?
\end{question}

In classical computational algebra, Gr\"obner bases do more than simply answer ideal membership questions.  They also are used as input by other algorithms to find invariants describing the underlying geometry and algebra such as dimension, degree, Hilbert series, etc.  

\begin{question}
What is a good notion of the variety defined by an $\SymN$-invariant ideal (of an infinite-dimensional ring)? How should one define its dimension? 
\end{question}

\begin{question}
Is there a better (alternative) notion of Hilbert series, one that would be suitable for $\Inc$-invariant modules? 
(See issues discussed in the last section of~\cite{krone2016hilbert}.) 
\end{question}

% Anton: I am not sure this relates to EGBs. The "structure constants" part is vague: just get more bounds for whatever?
% Another exciting area of combinatorial asymptotic algebra research is the computation of degree bounds for various symmetric ideals. 
% For instance, in \cite{draisma2014bounded, Draisma12f} a universal degree bound for generators (up to symmetry) cutting out tensors of a given border rank was shown to exist.  However, explicit bounds for this degree are typically only known in special cases.  Other questions arise from the study of special cases of toric chains.  For instance, prove what are the numbers in the tables in \cite{KKL:equivariant-markov, Hillar13, hillar2016corrigendum}, which enumerate minimal degrees for generators of the infinite chain.

% \begin{problem}
% Prove bounds for structure constants of chains of ideals.
% \end{problem}

While it is easily observed that, in practice, computations of EGBs tend to consume far more resources than in the classical case (per bit of input), there is no good understanding of theoretical complexity of an equivariant Buchberger's algorithm.

\begin{question}
Given widths and degrees of a finite set of generators, is there an upper bound on widths and degrees of the elements of a reduced EGB?

If so, then one could look for lower bounds (in the worst case). 
\end{question}

% Anton: too vague.
% Most of the considerations here have involved specific monoids, such as the monoid of increasing functions or the symmetric group.  It is natural to study other actions on ideals.

% \begin{problem}
% Study equivariant Gr\"obner bases for other monoids.
% \end{problem}

% Question of finiteness in non-commutative settings have inspired workers in Gr\"obner bases from early days \cite{shirshov1962some, bokut1976embeddings, bergman1978diamond} to recent times \cite{drensky2006grobner, la2009letterplace}.

% \begin{problem}
% Develop practical non-commutative extensions to equivariant Gr\"obner bases.
% \end{problem}

% With these and other fundamental problems still unresolved, the field of asymptotic algebra has many productive years ahead of itself.

One of the largest computations done so far is that of~\cite{Brouwer09e}; it is accomplished by a custom made program (not available publicly). The output gives a definitive algebraic-statistical description of the Gaussian two-factor model by means of EGBs.  We propose the following difficult challenge.

\begin{problem}
Use EGBs to study the Gaussian three-factor model; i.e., obtain the kernel of the map:
\begin{align*} 
\mathbb C[y_{ij} \mid i,j \in \mathbb N, i > j] &\to \mathbb C[s_i,t_i,u_i \mid i \in \mathbb N],\\ 
y_{ij} &\mapsto s_is_j + t_it_j + u_iu_j\,.
\end{align*}
\end{problem} 

While this may be set up exactly with the same technique as in~\cite{Brouwer09e}, the computation seems to present an insurmountable task for the current implementations of current algorithms executed on current computers.

%Anton: speculation?
% yes, but there is some evidence.  better?  -cjh
% Anton: not convinced. the problem is vague. best to finish with a concrete challenge.
 
% Finally, we close with speculation that an approach using tools from the representation theory of the symmetric group could be productive.  For instance, one might develop further the application of invariant modules \cite{camina1991some} to lattice ideals \cite{Hillar13} or the use of representation theory to study finite generation of ideals \cite{kemer2008analog}.

%  \begin{problem}
%  Develop tools from representation theory to understand the structure of equivariant ideals.
%  \end{problem}

\bibliographystyle{plain}
\bibliography{egb}

\end{document}